\documentclass[12pt,a4paper]{amsart}
\usepackage{amsmath}
\usepackage{extarrows}
\usepackage{amssymb,amsthm,amscd}
\usepackage{fullpage}
\usepackage{hyperref}
\usepackage{comment}
\usepackage{mathtools}
\usepackage[normalem]{ulem} 
\usepackage{amsthm}
\usepackage{breqn}

\usepackage{subcaption}
\usepackage{graphicx}
\usepackage{tikz}
\usepackage{tikz-cd}
\usetikzlibrary{matrix,arrows,positioning,calc,chains}

\usepackage[draft,multiuser,english]{fixme}

\fxsetup{layout=inline}
\FXRegisterAuthor{p}{ap}{P\!\!}

\FXRegisterAuthor{r}{ar}{R\hspace{-.5em}}

\newcommand{\name}[1]{\textsc{#1\/}}

\DeclareMathOperator{\Aff}{Aff}
\DeclareMathOperator{\SAff}{SAff}
\DeclareMathOperator{\J}{J}
\DeclareMathOperator{\jac}{jac}

\usepackage{enumerate}

\address{\noindent Institut f\"{u}r Mathematik, Friedrich-Schiller-Universit\"{a}t Jena,   Jena 07737, Germany}
\email{andriyregeta@gmail.com}

\newtheorem{claim}{Claim}

\theoremstyle{plain}
\newtheorem{theorem}{Theorem}[section]
\newtheorem{lemma}[theorem]{Lemma}

\newtheorem{proposition}[theorem]{Proposition}

\theoremstyle{definition}
\newtheorem{definition}[theorem]{Definition}

\newtheorem{remark}[theorem]{Remark}

\def\Aut{\operatorname{Aut}}

\def\Tr{\operatorname{Tr}}

\def\Bir{\operatorname{Bir}}
\def\SAut{\operatorname{SAut}}

\def\Der{\operatorname{Der}}

\def\alg{\operatorname{alg}}

\def\deg{\operatorname{deg}}

\def\id{{\mathrm{id}}}

\def\ZZ{{\mathbb Z}}
\def\CC{{\mathbb C}}

\def\T{{\mathrm T}}
\def\QQ{{\mathbb Q}}
\def\NN{{\mathbb N}}

\def\K{{\mathbb K}}
\def\G{{\mathbb G}}

\def\O{{\mathcal O}}

\def\embed{\hookrightarrow}

\def\Aff{\operatorname{Aff}}

\def\GL{\operatorname{GL}}
\def\PGL{\operatorname{PGL}}

\def\0{\circ}

\def\g{{\mathfrak g}}

\begin{document}

\title[When is the automorphism group of an affine variety linear?]{When is the automorphism group of an affine variety linear?}

\author{Andriy Regeta}

\begin{abstract} 
   Let  
$\mathrm{Aut}_{\text{alg}}(X)$ be the subgroup  of the group 
of regular automorphisms $\mathrm{Aut}(X)$ of an  affine algebraic variety $X$
 generated by
all  connected algebraic subgroups. We prove that  if $\dim X \ge 2$ and if $\mathrm{Aut}_{\text{alg}}(X)$ is ``rich enough'',  $\mathrm{Aut}_{\text{alg}}(X)$ is not linear, i.e., it cannot be embedded into $\GL_n(\K)$, where $\K$ is an algebraically closed field of characteristic zero. Moreover, $\Aut(X)$ is isomorphic to an algebraic group as an abstract group only if the connected component of
$\Aut(X)$
is either the algebraic
torus or a direct limit of commutative unipotent groups.
Finally, we prove that for an uncountable  $\K$  the group of birational transformations of $X$ cannot be isomorphic to the group of automorphisms of an affine variety if $X$ is endowed with a rational action of a positive-dimensional linear algebraic group.
\end{abstract}

\maketitle

\section{Introduction}

In this paper we work over algebraically closed  field $\K$ of characteristic zero, and $X$ always denotes an irreducible affine variety.
It is well-known that the automorphism group of an affine variety may be very large. For example, the automorphism group $\Aut(\mathbb{A}^2)$ of  the affine plane $\mathbb{A}^2$ contains a free product of two polynomial rings in one variable.  Consequently, $\Aut(\mathbb{A}^2)$ is infinite-dimensional and cannot be given a structure of an algebraic group.
Moreover,
 it is shown in \cite[Proposition 2.3]{Cor} that $\Aut(\mathbb{A}^2)$ is not linear, i.e. $\Aut(\mathbb{A}^2)$ cannot be embedded into the general linear group $\GL_n(\K)$ as an abstract group.
The first main result of the present note is a  generalization of this statement to a big family of affine varieties. 

 It is well-known (Proposition \ref{ind-group}) that the automorphism group $\Aut(X)$ has a structure of an  \emph{ind-group} (see Section \ref{Ind-groups} for the definition) and if $\dim X \ge 2$, $\Aut(X)$ is infinite-dimensional unless $\Aut(X)$ is a countable extension of the algebraic torus.
 But even if the  automorphism group $\Aut(X)$  is infinite-dimensional it may happen that $\Aut(X)$  embeds into $\GL_n(\K)$. For exmaple, \cite[Example 6.14]{LRU} shows that there is an affine surface $S$ such that $\Aut(S)$ is isomorphic to the polynomial ring in one variable $\K[t]$  and as an abstract additive group $\Aut(S)$ is isomorphic to the additive group of the base field and hence embeds into $\GL_2(\K)$.
However, if $\Aut(X)$ is rich enough, $\Aut(X)$ cannot be embedded into $\GL_n(\K)$. More precisely, we prove the following statement.
 
 We denote the additive and 
 multiplicative group of the field $\K$ by $\mathbb{G}_a$ and $\mathbb{G}_m$ respectively.
  For a given  affine variety $X$  we denote by $\Aut_{\alg}(X)$  the subgroup of $\Aut(X)$ generated by all connected algebraic subgroups.

\begin{theorem}\label{main}
Assume  dimension of $X$ is $\ge 2$
and $X$ is
not  isomorphic to  $\mathbb{A}^1 \times Y$, where  $Y$ is an affine variety that admits no faithful action of positive-dimensional connected algebraic group. 
If   $\Aut_{\alg}(X)$ is non-commutative and contains a copy of $\mathbb{G}_m$,  then $\Aut_{\alg}(X)$
cannot be embedded into
 $\GL_n(\K)$.
 \end{theorem}
 
  If $X$ is isomorphic to $\mathbb{A}^1 \times Y$ for some affine variety $Y$, where $Y$ does not admit a faithful action of $\mathbb{G}_m$ or $\mathbb{G}_a$, then $\Aut_{\alg}(X)$ is isomorphic to $\mathcal{O}(Y)^* \ltimes \mathcal{O}(Y)^+$ which can be embedded into 
$\mathbb{K}(Y)^* \ltimes \mathbb{K}(Y)^+$ which in turn  embeds into $\overline{\mathbb{K}(Y)}^* \ltimes \overline{\mathbb{K}(Y)}^+ \simeq \mathbb{K}^* \ltimes \mathbb{K}^+$.

If $X$ admits no $\mathbb{G}_m$-action, but admits two non-commuting $\mathbb{G}_a$-actions, $\Aut(X)$ can be embedded into $\GL_n(\mathbb{K})$.
For example,  there exists an affine surface $X$ (see \cite[Example 4.1.3]{BD}) that has an automorphism group $\Aut(X) = \Aut_{\alg}(X) \simeq \mathbb{K}[x] \ast \mathbb{K}[y]$ which is linear by \cite[Theorem]{Mar} as additive groups $\mathbb{K}[x] \simeq \mathbb{K}[y]$ are isomorphic as abstract groups  to $\mathbb{G}_a$.  
 
 The main idea of the proof of Theorem \ref{main} is to find a subgroup of $\Aut_{\alg}(X)$  that is isomorphic to 
  a direct limit of  subgroups isomorphic to $\G_m \ltimes \G_a^r$, $r \ge 2$, where $\G_a^r$ is a direct product of root subgroups with respect to $\mathbb{G}_m$ of different weights, and show that such a group cannot be embedded into $\GL_n(\K)$.

As an application of Theorem \ref{main}  we obtain that 
the  automorphism group $\Aut(X)$ is isomorphic to a linear algebraic group as an abstract group if and only if the connected component $\Aut^\circ(X) \subset \Aut(X)$  is commutative (see Theorem \ref{main2}). Moreover, in this case $\Aut^\circ(X)$ is either the algebraic torus or a direct limit of commutative unipotent groups.

We denote by $\Bir(X)$  the group of birational transformations of $X$. It is well-known that such a group may be very large. For example the Cremona group $\Bir(\mathbb{A}^n) = \Bir(\mathbb{P}^n)$  for $n>1$
is known  to be very big, in particular, much larger than $\Aut(\mathbb{A}^n)$.  
Proposition \ref{Cremona} shows that the Cremona group $\Bir(\mathbb{A}^n) = \Bir(\mathbb{P}^n)$, $n>0$, is not isomorphic to the automorphism group of any affine variety. Moreover,
if $\Bir(X)$ is ``rich enough'', $\Bir(X)$ is also not isomorphic to the automorphism group of any affine variety. 
More precisely,
we prove the following result.

\begin{theorem}\label{main3}
Assume $\mathbb{K}$ is uncountable
and $X,Y$ are affine irreducible algebraic varieties. Assume $X$ is endowed with a rational action of a positive-dimensional linear algebraic group. Then the group of birational transformations $\Bir(X)$ is not isomorphic to $\Aut(Y)$.
\end{theorem}

\emph{Acknowledgements}: I would like to thank  Ievgen Makedonskyi for useful comments.

\section{Preliminaries}

\subsection{ Derivations and group actions}
Recall that $X$ is 
an irreducible affine algebraic variety.
A derivation $\delta$ is called \emph{locally finite} if it acts locally finitely on $\O(X)$, i.e., for any function $f \in \O(X)$ there is a  finite-dimensional vector subspace
$W \subset \O(X)$ such that $f \in W$ and  $W$ is stable under action of $\delta$.
A derivation $\delta \in \Der(\O(X))$ is called \emph{locally nilpotent} if for any function $f \in \O(X)$ there exists $k \in \mathbb{N}$ (which depends on $f$)
such that $\delta^k(f) = 0$. Note that
there is a one-to-one correspondence between 
locally nilpotent derivations on $\O(X)$ and
 $\G_a$-actions on 
$X$ given by the map $\delta\mapsto\{t\mapsto\exp(t\delta)\}$.

An element $u\in\Aut(X)$ is called \emph{unipotent} if $u=\exp(\partial)$ for some locally nilpotent derivation $\partial$.

\subsection{Ind-groups}\label{Ind-groups}

The notion of an ind-group goes back to \name{Shafarevich} who called it an infinite dimensional algebraic group (see \cite{Sh66}). 
We refer to \cite{FK}   for basic notions in this context.

\begin{definition}
By an affine \emph{ind-variety}  we understand an injective limit $V=\varinjlim V_i$ of an ascending sequence $V_0 \embed V_1 \embed V_2 \embed \ldots$ 
such that the following holds:
\begin{enumerate}[(1)]
\item $V = \bigcup_{k \in \NN} V_k$;
\item each $V_k$ is an affine algebraic variety;
\item for all $k \in \NN$ the embedding $V_k \embed V_{k+1}$ is closed in the Zariski topology.
\end{enumerate}
\end{definition}

For simplicity we will  call an affine ind-variety  simply an ind-variety.

An ind-variety $V$ has a natural \emph{topology}: a subset $S \subset V$ is called closed, resp. open, if $S_k := S \cap  V_k \subset V_k$ is closed, resp. open, for all $k \in \mathbb{N}$.  A  closed subset $S \subset V$ has a natural structure of an ind-variety and is called an ind-subvariety.  

A set theoretical product of ind-varieties admits a natural structure of an ind-variety. 
A \emph{morphism} between ind-varieties $V = \bigcup_m V_m$ and $W = \bigcup_l W_l$ is a map $\phi: V \to W$ such that for every $m \in \mathbb{N}$ there is an $l \in \mathbb{N}$ such that 
$\phi(V_m) \subset W_l$ and that the induced map $V_m \to W_l$ is a morphism of algebraic varieties. 
 This allows us to give the following definition.

\begin{definition}
An ind-variety $H$ is said to be an \emph{ind-group} if the underlying set $H$ is a group such that the map $H \times H \to H$, defined by  $(g,h) \mapsto gh^{-1}$, is a morphism of
ind-varieties.
\end{definition}

A \emph{closed subgroup}  $G$ of $H$ is a subgroup that is at the same time a closed subset. In this case $G$ is again an ind-group with respect to the induced ind-variety structure.
 A closed subgroup $G$ of an ind-group $H = \varinjlim H_i$ is called an \emph{algebraic subgroup} if $G$ is contained in $H_i$ for some $i$.

The next result can be found in   \cite[Section 5]{FK}.
\begin{proposition}\label{ind-group}
Let $X$ be an affine variety. Then $\Aut(X)$ has the structure of an ind-group such that  a regular action of an algebraic group $H$ on $X$ induces an ind-group homomorphism $H \to \Aut(X)$.
\end{proposition}

\subsection{Root subgroups}\label{rootsubgroups}
	In this section we describe  {\it root subgroups} of
	$\Aut(X)$ for a given affine variety $X$ with respect to a subtorus.
	
	\begin{definition}
		Let $T$ be a subtorus in $\Aut(X)$, i.e. a closed
		algebraic subgroup isomorphic to a torus.  A closed algebraic subgroup
		$U \subset \Aut(X)$ isomorphic to $\mathbb{G}_a$ is called a {\it root
			subgroup} with respect to $T$ if the normalizer of $U$ in
		$\Aut(X)$ contains $T$.
		
		Since $\mathbb{G}_a$ contains no non-trivial closed normal subgroups, every non-trivial regular action is faithful. Hence, such an algebraic
		subgroup $U$ corresponds a non-trivial {\it normalized} $\mathbb{G}_a$-action on $X$,
		i.e.\,a $\mathbb{G}_a$-action on $X$ whose image in $\Aut(X)$ is normalized
		by $T$.
	\end{definition}
	
	Assume $U\subset \Aut(X)$ is a root subgroup with respect to $T$. Since
	$T$ normalizes $U$, we can define an action
	$\varphi\colon T\rightarrow \Aut(U)$ of $T$ on $U$ given by
	$t.u=t\circ u\circ t^{-1}$ for all $t\in T$ and $u\in U$.
	Moreover, since $\Aut(U)\simeq \mathbb{G}_m$, such an action corresponds to a
	character of the torus $\chi\colon T\rightarrow \mathbb{G}_m$, which does not
	depend on the choice of isomorphism between $\Aut(U)$ and
	$\mathbb{G}_m$. This character is
	called the \emph{weight} of $U$. The algebraic subgroups $T$ and $U$ generate an algebraic subgroup in $\Aut(X)$	isomorphic to $\mathbb{G}_a\rtimes_\chi T$.

 Consider a nontrivial algebraic action of $\mathbb{G}_a$ on $X$, given by $\lambda \colon \mathbb{G}_a \to \Aut(X)$. If $f \in \mathcal{O}(X)$ is $\mathbb{G}_a$-invariant, then the \emph{modification} $f \cdot \lambda$ of $\lambda$ is defined in the    following way:
$$(f \cdot \lambda)(r)x =\lambda(f(x)r)x$$ 
for $r \in \mathbb{C}$ and $x \in X$.
 This is again a $\mathbb{G}_a$-action.  
 It is not difficult to see that if $X$ is irreducible and $f  \neq 0$, then $f \cdot \lambda$  and $\lambda$ have the same invariants.
If $U \subset \Aut(X)$ is a closed algebraic subgroup isomorphic to $\mathbb{G}_a$ and if $f \in \mathcal{O}(X)^U$ is a $U$-invariant, then in a similar way we define  the modification $f \cdot U$ of $U$. Pick an isomorphism
$\lambda \colon \mathbb{G}_a \to U$ and set 
$$f \cdot U = \{ (f \cdot \lambda)(r) \mid r \in \mathbb{G}_a \}.$$ 

\subsection{Divisible elements}\label{divisibleelements}
We call an element $f$ in a group $G$  \emph{divisible by $n$} if there exists an element $g \in G$ such
that $g^n = f$. An element is called \emph{divisible} if it is divisible by all $n \in \mathbb{Z}^+$. 
If $G$ is an agebraic group, then by  \cite[Lemma 3.12]{LRU}
for any $f \in G$ there exist $k > 0$ that depends on $f$ such
that $f^k$ is a divisible element.

\section{Proof of Theorem \ref{main}}
The following lemma is  well known  and appeared in similar form in \cite[Lemma 3.1]{FZ-uniqueness}. 
\begin{lemma}\label{l:hom-lnd}
Assume that $\g$ is $\ZZ^r$-graded for $r>0$ and
consider a locally finite element $z\in\g$ that does not belong to the zero component $\g_0$. Then 
there exists a locally nilpotent homogeneous component of $z$ of non-zero weight.
\end{lemma}
\begin{proof}
Let us take the convex hull $P(z)\subset \ZZ^r\otimes\QQ$ of component weights of $z$.
Then for any non-zero vertex $v\in P(z)$ the corresponding homogeneous component is locally nilpotent.
The details are left to the reader.
\end{proof}

\begin{proof}[Proof of Theorem \ref{main}]
Assume first that $\Aut_{\alg}(X)$ contains a copy of $\G_a$ and $T=\G_m$. Let $\partial$ be a locally niplotent derivation corresponding to $\G_a$. By Lemma \ref{l:hom-lnd} there is a locally nilpotent derivation $\tilde{\partial} \in \Der \mathcal{O}(X)$ that is normalised by $T$. Denote by $U$ a $\G_a$-action on $X$ that corresponds to $\tilde{\partial}$. Hence, $T$ acts on the ring of invariants $\mathcal{O}(X)^U$.
\begin{claim}
There is a $T$-semi-invariant $f \in \mathcal{O}(X)^U$ of non-zero weight.
\end{claim}
Assume towards a contradiction that all invariants from $ \mathcal{O}(X)^U$ are also $T$-invariants.
This implies that  
\begin{equation}\label{TU}
\mathcal{O}(X)^{U} = \mathcal{O}(X)^{T \ltimes U} = (\mathcal{O}(X)^{U})^T.
\end{equation}
This is possible only if $T \ltimes U$ acts with at most one-dimensional orbits since otherwise, if $T \ltimes U$ acts with a two-dimensional orbit, then by \cite[Theorem 11.1.1.(7)]{FK} the quotient field of $\mathcal{O}(X)^{T\ltimes U}$ has transcendence  degree $\dim X -2$ which contradicts the fact that the transcendence  degree of the quotient field of $\mathcal{O}(X)^{U}$ is $\dim X -1$. Since $T \ltimes U$ is a connected algebraic group, these orbits are irreducible  which are then isomorphic to $\mathbb{A}^1$ as $U$-orbits are closed subvarieties isomorphic to $\mathbb{A}^1$, see \cite[Theorem 11.1.1]{FK}.
We claim that $\mathcal{O}(X)^{T} =\mathcal{O}(X)^{U}$. Indeed, by \eqref{TU} we have the inclusion
$\mathcal{O}(X)^{U} \subset \mathcal{O}(X)^{T}$ and assume towards a contradiction that there is $f \in \mathcal{O}(X)^{T} \setminus \mathcal{O}(X)^{U}$.
In another words, $f(t.x) = f(x)$ for all $x \in X$, $t \in T$, but there exists $x \in X$ such that $f(u.x) \neq f(x)$, where $u \in U$ is a non-trivial element. Hence, $U.x \simeq \mathbb{A}^1$. The torus $T$ acts on $U.x$ since otherwise $T \ltimes U.x$ would be two-dimensional. Moreover, $T$ acts on $U.x$ non-trivially since otherwise for any $y \in U.x$, $f(\overline{T.y}) = f(y)$ which implies that $f(tut^{-1}.y) = f(u.y)$. But this contradicts the fact that $U$ acts on $U.x$ transitively since $tut^{-1} \neq u$. Hence, there is a quotient morphism
\[
X \to X /\!\!/ T = X /\!\!/ U  
\]
Its fibers are at least one-dimensional and contain a unique closed $T$-orbit. Hence, fibers are one-dimensional and  moreover, are isomorphic to $U$. By \cite[Proposition 3.9.1]{Kr}
$X$ is isomorphic to $\mathbb{A}^1 \times Y$, where $Y$ is an affine variety isomorphic to $X /\!\!/ T \simeq X /\!\!/ U$ which contradicts the assumption on $X$.

\vspace*{3mm}

Therefore, 
$\{  f^k \cdot U \subset \Aut(X) \mid k \in \mathbb{N} \}$ are root subgroups with respect to $T$ with different weights. Without loss of generality we can assume that $U$ is a root subgroup with respect to $T$ of non-zero weight since otherwise we can just replace $U$ by $f \cdot U$.

\begin{claim}\label{claimnotlinear}
The subgroup \[
G=T \ltimes (\bigoplus_{k \ge 1} f^k \cdot U) \subset \Aut(X)
\]
is not linear.
\end{claim}
Indeed,  assume towards a contradiction that the subgroup $G \subset \Aut(X)$ is linear, i.e., there is an embedding $\varphi\colon G \to \GL_n(\K)$. Since $G$ is solvable, $G$ can be embedded (after a necessarily conjugation) into a Borel subgroup $B \subset \GL_n(\K)$  of upper triangular matrices. Moreover,  the commutator $[G,G] = \bigoplus_{k \ge 1} f^k \cdot U$ embedds into $[B,B]$. In other words 
$\varphi(\bigoplus_{k \ge 1} f^k \cdot U)$ is a subgroup of the unipotent radical of $B$.

Consider the closed subgroup $\overline{\varphi(T)}^\circ \ltimes \overline{\varphi(f^k \cdot U)} \subset B \subset \GL_n(\K)$.  The subgroup
$\overline{\varphi(f^k \cdot U)} = \overline{\varphi(f^k \cdot U)}^\circ \subset [B,B]$  is  unipotent and $\overline{\varphi(T)} \subset B$ is an algebraic subgroup. Hence, $\overline{\varphi(T)}^\circ \subset \overline{\varphi(T)}$ is a finite index subgroup which implies that $\overline{\varphi(T)}^\circ$ contains infinitely many  elements of finite order of $\varphi(T)$. As a consequence, $\overline{\varphi(T)}^\circ$  contains a copy of algebraic subtorus  of positive dimension.
Pick a big enough $k \in\mathbb{N}$ such that 
  the kernel of $T$-action on $f^k \cdot U$ is
$\langle \xi_k  \rangle$, where  $\xi_k$ is an element of  order
bigger than the index $s=[\overline{\varphi(T)}:\overline{\varphi(T)}^\circ]$  and 
    $\xi_k$  acts on  $\mathbb{K}f$  non-trivially.
Hence, $\xi_k^{s} \in \overline{\varphi(T)}^\circ$ and since $k$ is chosen to be big enough,  we have that
\begin{equation}\label{non-trivially}
    \xi_k^{s} \text{ acts on } \mathbb{K}f \text{ non-trivially.}
\end{equation}
Since   
$\varphi(\xi_k^{s})$ centralizes $\varphi(f^k \cdot U)$, $\varphi(\xi_k^{s})$ centralizes $\overline{\varphi(f^k \cdot U)}$ too.
Choose a  subtorus of $\overline{\varphi(T)}^\circ$ which we denote by $\tilde{T}$ that contains $\varphi(\xi_k^s)$. Pick  $u_k \in \varphi(f^k \cdot U)$ and consider
the unipotent subgroup $V_k =\langle \tilde{T}.u_k \rangle = \langle tu_k t^{-1} \mid t \in \tilde{T} \rangle \subset \GL_n(\mathbb{K})$. 
Note that $\tilde{T}$ normalizes $V_k$.
Hence, the unipotet group $V_k$ is a direct product of root subgroups with respect to $\tilde{T}$.
The kernel of $\tilde{T}$-action on $V_k$ contains $\langle \varphi(\xi_k^s) \rangle$. 
Since $\GL_n(\mathbb{K})$ is an algebraic group, i.e., is finitely dimensional, for a big enough $k$, the weights of all root subgroups  of $V_k$ with respect to $\tilde{T}$ are the same as the weights of the root subgroups of $V_{k+1} = \langle \tilde{T}.u_{k+1} \rangle$, where $u_{k+1} \in \varphi(f^{k+1}\cdot U) \subset [B,B] \subset \GL_n(\K)$. Hence, $\langle \varphi(\xi_k^s) \rangle$ acts trivially on $V_{k+1}$.
As a consequence, $\langle \xi_k^s \rangle$ acts trivially on $\varphi^{-1}(V_k) \subset f^k \cdot U$ and on  $\varphi^{-1}(V_{k+1}) \subset f^{k+1} \cdot U$. Therefore,  $\langle \xi_k^s \rangle$ acts trivially on $\overline{\varphi^{-1}(V_{k})} \subset f^k \cdot U$ and on $\overline{\varphi^{-1}(V_{k+1})} \subset f^{k+1} \cdot U$ which implies that $\langle \xi_k^s \rangle$ acts trivially on  $\mathbb{K}f$.  This contradicts \eqref{non-trivially} which proves the theorem if $X$ admits $\mathbb{G}_m$- and $\mathbb{G}_a$-actions.

 If $X$ admits two non-commuting $\mathbb{G}_m$-actions, then by Lemma \ref{l:hom-lnd} $X$ admits a $\mathbb{G}_a$-action and the claim of the theorem follows from above.
\end{proof}

\begin{remark}
As it is already mentioned in the introduction, it is proved in \cite{Cor} that  $\Aut(\mathbb{A}^2)$ is not linear, i.e., it cannot be embedded into $\GL_n(\mathbb{K})$ for any $n \in \mathbb{N}$. This also follows from Theorem \ref{main}. Moreover, in \cite[Proposition 2.3]{Cor} it is proved that there is a countably generated subgroup of the subgroup
\[\J = \{ (ax+c, by + f(x)) | a,b \in \mathbb{C}^*, c \in \mathbb{C}, f(y) \in \mathbb{C}[x] \} \subset \Aut(\mathbb{A}^2)\]
that is not linear. We note that by the Jung-Van der Kulk Theorem (see \cite{Ju42} and \cite{Kul53}) $\Aut(\mathbb{A}^2)$ is the amalgamated product of $\J$ and the group of affine transformations $\Aff_2$ of $\mathbb{A}^2$ along their intersection $C$, i.e.,
\begin{equation}\label{formula}
\Aut(\mathbb{A}^2) = \Aff_2 \ast_C \J.
\end{equation}
Using such an amalgamated product we claim that any representation of a subgroup \[
\SAut(\mathbb{A}^2) = \bigg\{ (f,g) \in \Aut(\mathbb{A}^2) \mid  \jac(f) = \det \begin{pmatrix}
\frac{\partial f}{\partial x} & \frac{\partial f}{\partial y} \\
\frac{\partial g}{\partial x} & \frac{\partial g}{\partial y}
\end{pmatrix} = 1 \bigg\}
\]
is trivial, i.e., any homomorphism $\varphi\colon \SAut(\mathbb{A}^2) \to \GL_n(\mathbb{K})$ is trivial. To show this we first note that the amalgamated product structure of $\Aut(\mathbb{A}^2)$ induces the amalgamated product structure of $\SAut(\mathbb{A}^2)$. More precisely, $\SAut(\mathbb{A}^2)$ is the amalgamated product of the group $\SAff_2$ of special affine transformations of $\mathbb{A}^2$ and the subgroup 
\[\J^s = \{ (ax+c, by + f(x)) | a,b \in \mathbb{C}^*, ab=1, c \in \mathbb{C}, f(y) \in \mathbb{C}[x] \} \subset \Aut(\mathbb{A}^2).\]
  By \cite[Proposition 2.3]{Cor} there is no embedding of $\J^s$ into $\GL_n(\mathbb{K})$. Hence, there is a non-identity element $g \in \J^s$ such that $g$ is the kernel of $\varphi$. Therefore, the normal subgroup that contains $g$ is also in the kernel of $\varphi$. But by \cite{FL} any normal subgroup that contains
$g$ coincides with $\SAut(\mathbb{A}^2)$ which proves the claim.

Moreover, any group homomorphism $\psi\colon \Aut(\mathbb{A}^2) \to \GL_n(\mathbb{K})$ factors through the homomorphism $\jac\colon \Aut(\mathbb{A}^2) \to \mathbb{G}_m$.
  Indeed, similarly as above, there is $g \in \J^s$ that is in the kernel of $\psi$. Hence, by the same argument as above the normal subgroup generated by $g$ contains $\SAut(\mathbb{A}^2)$ and the claim follows.
\end{remark}

\section{Automorphism group that is isomorphic to a linear algebraic group}
We begin this section with the lemma that is used in the proof of Theorem \ref{main2}.
\begin{lemma}\label{A2}
Let 
$\tilde{U},\tilde{V} \subset \Aut(\mathbb{A}^2)$ be
two  one-dimensional unipotent subgroups that act on $\mathbb{A}^2$ with different generic orbits. Then
\begin{enumerate}[(1)]
\item\label{(1)} the subgroup $G_{\tilde{U}} \subset \Aut(\mathbb{A}^2)$ generated by all one-dimensional unipotent subgroups that have the same generic orbits as $\tilde{U}$ coincides with its centralizer;
\item\label{(2)} the subgroup generated by 
$G_{\tilde{U}}$ and $G_{\tilde{V}}$ cannot be presented as a finite product of $G_{\tilde{U}}$ and $G_{\tilde{V}}$.
\end{enumerate}
\end{lemma}
\begin{proof}
Recall that the group $\Aut(\mathbb{A}^2)$ has the amalgamated product structure $\Aff_2 \ast_{C} \J$, see \eqref{formula}. 
 By \cite{Sr80} any closed algebraic subgroup of $\Aut(\mathbb{A}^2)$
 is conjugate to one of the factors $\Aff_2$ or $\J$. Since  $G_{\tilde{U}} \subset \Aut(\mathbb{A}^2)$ is infinite-dimensional, $G_{\tilde{U}}$ is conjugate to a subgroup of $\J$. Moreover, since $G_{\tilde{U}}$ is infinite-dimensional, commutative and consists of unipotent elements, $G_{\tilde{U}}$ is conjugate to a subgroup   of
 \[ \J_u= \{ (x, y + f(x)) \mid f(y) \in \mathbb{C}[x] \} \subset \Aut(\mathbb{A}^2).
 \] 
 Since $G_{\tilde{U}}$  is generated by all one-dimensional unipotent subgroups
that have the same generic orbits,  $G_{\tilde{U}}$ is conjugate to the whole $\J_u$. 
It is easy to check that $\J_u \subset \Aut(\mathbb{A}^2)$ coincides with its centralizer which proves \eqref{(1)}. 

Without loss of generality we can assume that $G_{\tilde{U}} = \J_u$. Since by \eqref{(1)}  $G_{\tilde{V}}$ does not commute with $G_{\tilde{U}} =\J_u$, $G_{\tilde{V}}$ is not a subgroup of $\J_u$ and since $G_{\tilde{V}}$ is infinite-diemnsional, $G_{\tilde{V}}$ is not a subgroup of $\J$. Hence,  \eqref{(2)} follows from amalgamated product structure of $\Aut(\mathbb{A}^2)$.
\end{proof}

\begin{theorem}\label{main2}
Let $X$ be an affine variety.
If $\Aut(X)$ is isomorphic to a linear algebraic group as an abstract group, then the connected component $\Aut^\circ(X)$ is commutative. Moreover, in this case $\Aut^\circ(X)$ is either the algebraic torus or a direct limit of commutative unipotent groups.
\end{theorem}
\begin{proof}[Proof of Theorem \ref{main2}]
Assume $\varphi\colon G \to \Aut(X)$ is an isomorphism of abstract groups, where $G$ is a  linear algebraic group. 
If the connected component $G^\circ \subset G$ is commutative, the finite index subgroup of $\Aut(X)$ is commutative. This implies that the connected component $\Aut^\circ(X)$ is commutative.
Hence,  by \cite[Theorem B]{CRX19} (see also \cite[ Corollary 3.2]{RvS21})
$\Aut(X)^\circ$ is a union of commutative algebraic groups.  The group $\Aut(X)^\circ$ either does not contain unipotent subgroups and in this case $\Aut^\circ(X)$ is isomorphic to an algebraic torus  or $\Aut(X)^\circ$ contains unipotent subgroups.  
By Theorem \ref{main} in the  later  case either $\Aut(X)$ does not contain a copy of $\mathbb{G}_m$ which implies that $\Aut^\circ(X)$ is the union of unipotent algebraic subgroups or
$X$ is isomorphic to $\mathbb{A}^1 \times Y$ for some affine variety $Y$. 
But  $\Aut(\mathbb{A}^1 \times Y)$ is non-commutative. This proves the theorem in case $G^\circ$ is commutative.

We  assume now that $G^\circ$ is non-commutative. 
In this case 
 $G$ contains closed connected commutative subgroups $U$ and $V$  that do not commute and $U$ normalizes $V$. Indeed, if $G$ is non-unipotent, it contains a maximal subtorus $T \subset G$ and a root subgroup  normalized but not centralized by $T$.
 If
$G$ is unipotent,  then $G$ is nilpotent, i.e., 
\begin{equation}
    G = G_0 \rhd G_1 \rhd \dots \rhd G_n = \{  \id \},
\end{equation}
where $G_{i+1} = [G,G_i]$, $[G,G_i] = \{ ghg^{-1}h^{-1} \mid g \in G, h \in G_i    \}$. In particular, $G_{n-1}$ is a subgroup of the center of $G$. Moreover, for any $H \subset G_{n-2} \setminus G_{n-1}$ isomorphic to $\mathbb{G}_a$, the group $V=H \times G_{n-1}$ is commutative. Choose a subgroup $U \subset G \setminus V$  isomorphic to $\mathbb{G}_a$ that does not commute with $V$. Note that such $U$ exists as $G$ is non-commutative. Moreover, we claim that $U$ normalizes $V$. Indded,  $[U,V] \subset [G,G_{n-2}] =G_{n-1}$ which means that $uvu^{-1}v^{-1} \in G_{n-1}$ for any $u \in U$, $v \in V$.  Hence, $uvu^{-1} \in G_{n-1} v \subset V$ which proves the claim.

Hence, $\overline{\varphi(U)}^\circ, \overline{\varphi(V)}^\circ \subset \Aut(X)$ are  closed connected commutative subgroups. 
Since $\varphi(U)$ normalizes $\varphi(V)$, $\varphi(U)$ normalizes 
$\overline{\varphi(V)} \subset \Aut(X)$ and hence   $\overline{\varphi(U)}^\circ$ normalizes 
$\overline{\varphi(V)}^\circ$. 
By \cite[Theorem B]{CRX19} (see also \cite[Corollary 3.2]{RvS21}) $\overline{\varphi(U)}^\circ, \overline{\varphi(V)}^\circ \subset \Aut(X)$ are unions of algebraic subgroups and hence $\overline{\varphi(U)}^\circ \ltimes \overline{\varphi(V)}^\circ$ is a union of algebraic subgroups. Note that $\overline{\varphi(U)}^\circ$ and $\overline{\varphi(V)}^\circ$ do not commute
since otherwise the subgroups $\varphi^{-1}(\overline{\varphi(U)}^\circ)$ and 
$\varphi^{-1}(\overline{\varphi(V)}^\circ)$ of $G$  would commute which is not the case as $\varphi^{-1}(\overline{\varphi(U)}^\circ) \cap U \subset U$ is a dense subgroup and analogously $\varphi^{-1}(\overline{\varphi(U)}^\circ) \cap U \subset U$ is a dense subgroup.
Therefore, there are non-commuting algebraic subgroups in $\Aut(X)$.
\begin{claim}
$\Aut_{\alg}(X) \subset \Aut(X)$ does not contain a copy of $\mathbb{G}_m$.
\end{claim}

 Since $\Aut_{\alg}(X)$ is not commutative,   by Theorem \ref{main}
 $\Aut_{\alg}(X)$ can contain a copy of $\mathbb{G}_m$ only if  $X$ is isomorphic to a product $\mathbb{A}^1 \times Y$, where $Y$ is an affine variety that does not admit a regular action of a positive-dimensional algebraic group. So, assume $X \simeq \mathbb{A}^1 \times Y$ and $Y \subset \mathbb{A}^l$ is a closed subset. The automorphism group $\Aut(\mathbb{A}^1\times Y)$ has the following form:
  \begin{eqnarray}\nonumber
 \{ (x,y_1,\dots,y_l) \mapsto & 
 (g(y)x + h(y), h_1(x,y),\dots,h_l(x,y)) \mid y=(y_1,\dots,y_l) \in Y, \; \\ & 
 g(y) \in \mathcal{O}(Y)^*, h \in \mathcal{O}(Y), h_i\in  \mathcal{O}(\mathbb{A}^1 \times Y) \},
 \label{structureofaut}
 \end{eqnarray}
 where $i = 1,\dots,l$ and the map $Y \to Y, \; (y_1,\dots,y_l) \mapsto (h_1(x,y),\dots,h_l(x,y))$  is an automorphism for each $x \in \mathbb{A}^1$. Consider
 the one-dimensional algebraic subtorus
 \begin{align*}
T= \{ (x,y_1,\dots,y_l) \mapsto  (ax, y_1,\dots,y_l) \mid \;   a \in \mathbb{K}^* \} \subset \Aut(\mathbb{A}^1 \times Y)
 \end{align*}
and
 the subgroup 
  \begin{align*}
 \Aut^u(\mathbb{A}^1 \times Y) = \{ (x,y_1,\dots,y_l) \mapsto  (x + h(y), y_1,\dots,y_l) \mid  \;  h \in \mathcal{O}(Y)  \} \subset \Aut(\mathbb{A}^1 \times Y).
 \end{align*}
 
 Any automorphism of $\Aut(\mathbb{A}^1 \times Y)$ that commutes with each element of $T$ has the form 
 \begin{equation}
     (x,y_1,\dots,y_l) \mapsto  (g(y)x, h_1(x,y),\dots,h_l(x,y)),
 \end{equation}
 where $g\in \mathcal{O}(Y)^*$ and $h_i$ does not depend on the first coordinate.
 This follows from the following  equality
 \begin{align*}
(g(y)x, h_1(x,y),\dots,h_l(x,y))  =
 (a^{-1}x, y_1,\dots,y_l) & \circ (g(y)x, h_1(x,y),\dots,h_l(x,y)) \circ \\ (ax, y_1,\dots,y_l) = 
 &
 (g(y)x, h_1(ax,y),\dots,h_l(ax,y)).
 \end{align*}

Now, since $T$ acts on $\Aut^u(\mathbb{A}^1 \times Y)$ by conjugations, $\overline{\varphi^{-1}(T)}$ acts on $\overline{\varphi^{-1}(\Aut^u(\mathbb{A}^1 \times Y))}$  by conjugations. Note that $\overline{\varphi^{-1}(T)}$ contains an algebraic subtorus $\tilde{T}$ of positive dimension as $\varphi^{-1}(T)$ contains infinitely many elements of finite order. The image of $\tilde{T}$ does not contain elements of the form
 \[
 (x,y_1,\dots,y_l) \mapsto  (g(y)x, y_1,\dots,y_l),
 \]
 where $g(y) \in \mathcal{O}(Y)^*\setminus \mathbb{K}$ because such  elements are not divisible by high enough positive integer and all elements of the algebraic torus are divisible (see Section \ref{divisibleelements}).

 The subgroup $\tilde{T} \subset \overline{\varphi^{-1}(T)}$ acts on $\overline{\varphi^{-1}(\Aut^u(\mathbb{A}^1 \times Y))}$ by conjugations and for each 
 \begin{align*}
u_n = \{ (x,y_1,\dots,y_l) \mapsto  (x + h_n(y), y_1,\dots,y_l) \mid  \;  h_n \in \mathcal{O}(Y) \} \subset \Aut(\mathbb{A}^1 \times Y),
 \end{align*}
 where $\{ h_n  \}$ form a basis of the vector space $\mathcal{O}(Y)$,  the  subgroup $\langle \tilde{T}.\varphi^{-1}(u_n) \rangle = \langle t   \varphi^{-1}(u_n) t^{-1} \mid t \in \tilde{T} \rangle$  is an algebraic subgroup of $G$ as the orbit $\tilde{T}.\varphi^{-1}(u_n)= \{ t   \varphi^{-1}(u_n) t^{-1} \mid t \in \tilde{T} \} \subset G$ is constructible.  Moreover, by \eqref{structureofaut} and because each element of $\tilde{T}$ has the form 
 \begin{equation}
     (x,y_1,\dots,y_l) \mapsto  (ax, h_1(y),\dots,h_l(y)),
 \end{equation}
  the subgroup generated by $\langle \tilde{T}.\varphi^{-1}(u_n) \rangle$  is a subgroup of $\varphi^{-1}(u_n \cdot \Tr)$, where
  \begin{align*}
 \Tr=  \{ (x,y_1,\dots,y_l) \mapsto  (x + b, y_1,\dots,y_l) \mid  \;  b \in \mathbb{K} \} \subset \Aut(\mathbb{A}^1 \times Y).
 \end{align*}
Hence, for any natural $k$
 the subgroup generated by $\langle \tilde{T}.\varphi^{-1}(u_n) \rangle,\dots, \langle \tilde{T}.\varphi^{-1}(u_{n+k}) \rangle$ is an algebraic subgroup of dimension at least $k$.
 But this is not possible as $G$ is an algebraic group, i.e., is finite-dimensional. We arrive to the contadiction which proves the claim.

  The semidirect product $\overline{\varphi(U)}^\circ \ltimes \overline{\varphi(V)}^\circ$ is the  union of  unipotent subgroups and in particular it contains a unipotent subgroup $W$ that acts on $X$ with a two-dimensional orbit $O$ that is isomorphic to $\mathbb{A}^2$, see \cite[Theorem 11.1.1]{FK}.   
  Pick subgroups $\tilde{U} \subset W$ and $\tilde{V} \subset W$  isomorphic to $\mathbb{G}_a$ that generate the algebraic subgroup that acts with a two-dimensional orbit $O \simeq \mathbb{A}^2$. 
 Take the maximal commutative subgroups $H_1$ and $H_2$ of $\Aut(X)$ that contain $\mathcal{O}(X)^{\tilde{U}} \cdot \tilde{U}$ and  $\mathcal{O}(X)^{\tilde{V}} \cdot \tilde{V} \subset \Aut(X)$ respectively. Therefore, $\varphi^{-1}(H_1), \varphi^{-1}(H_2) \subset G$ are closed subgroups. Hence, the subgroup $H \subset G$ generated by $\varphi^{-1}(H_1)$ and $\varphi^{-1}(H_2)$ can be presented as a finite product of subgroups 
$\varphi^{-1}(H_1)$ and $\varphi^{-1}(H_2)$. On the other hand, the subgroup of $\Aut(X)$ generated by $H_1$ and $H_2$ cannot be presented as a finite product of subgroups $H_1$ and $H_2$. Indeed,  by Lemma \ref{A2}\eqref{(1)}
the restriction of $H_1$ to 
$O \simeq \mathbb{A}^2$ is the subgroup of $\Aut(O \simeq \mathbb{A}^2)$ generated by all one-dimensional unipotent subgroups
that have the same generic orbits as $\tilde{U}\mid_{O}$. Analogous situation we have with $H_2$. Now by Lemma \ref{A2}\eqref{(2)}, the subgroup
$H = \langle H_1,H_2 \rangle \subset \Aut(X)$ restricted to $O \simeq \mathbb{A}^2$ cannot be presented as  a finite product of $H_1$ and $H_2$ restricted to $O$.
 We arrive to the contradiction and finish the proof. 
\end{proof}

\section{Proof of Theorem \ref{main3}}
We start this section with the next proposition. 
\begin{proposition}\label{Cremona}
Assume $\mathbb{K}$ is uncountable
and $X$ is a connected affine variety. Then $\Aut(X)$ is not isomorphic to  the Cremona group $\Bir(\mathbb{A}^n) = \Bir(\mathbb{P}^n)$ as an abstract group for any $n >0$.
\end{proposition}
\begin{proof}
The proof of this statement is similar to the proof of Theorem A in \cite{CRX19}. We give some details here for the convenience of the reader.
Let   
\[
\Tr = \{ (x_1,\dots,x_n) \mapsto (x_1 + c_1,\dots, x_n + c_n) \mid c_i \in \K \} \subset \Aut(\mathbb{A}^n) \subset \Bir(\mathbb{A}^n)
\]
be the subgroup of all translations and 
$\Tr_i$ be the subgroup of translations of the $i$-th coordinate:
\begin{equation}
(x_1, \ldots, x_n) \mapsto (x_1,\dots,x_i + c,\dots,x_n),
\end{equation}
where $c$ in $\K$.  Let
$\T\subset \GL_n(\K)\subset \Aut(\mathbb{A}^n) \subset \Bir(\mathbb{P}^n)$ be the diagonal group (viewed as a maximal torus) and let
$\T_i$ be the subgroup of automorphisms  
\begin{equation}
(x_1, \ldots, x_n)\mapsto (x_1,\dots,ax_i,\dots,x_n),
\end{equation}
where $a\in \K^*$. A direct computation shows that {\sl{$\Tr$ (resp. $T$) coincides with its centralizer in 
$\Bir(\mathbb{A}^n) = \Bir(\mathbb{P}^n)$}}.  
Assume towards a contradiction that  there is an isomorphism $\varphi\colon \Bir(\mathbb{A}^n) \to \Aut(X)$  of abstract groups.
Similarly as in \cite[Lemma 5.2]{CRX19} the groups $\varphi(\Tr)$, $\varphi(\Tr_i)$, $\varphi(T)$ and $\varphi(T_i)$ are closed subgroups of
$\Aut(X)$ for all $i = 1,\dots,n$.
	Now the proof of \cite[Theorem A]{CRX19} implies that $X \simeq \mathbb{A}^n$ and $\varphi(T) \subset \Aut(X \simeq \mathbb{A}^n)$ is isomorphic to the $n$-dimensional algebraic torus. Assume $U \subset \PGL_{n+1}(\mathbb{K}) \subset \Bir(\mathbb{A}^n)$ is a root subgroup with respect to $T$. This means that $T$ acts on $U$ with two orbits. 
	Therefore, $\varphi(U) \subset \Aut(X)$ is a constructible subset which is a group. We conclude that $\varphi(U) \subset \Aut(X)$ is an algebraic subgroup.
Moreover, since $\PGL_{n+1}(\mathbb{K})$ is generated by its finitely many root subgroups $U$ with respect to $T$, $\varphi(\PGL_{n+1}(\mathbb{K}))$ is generated by finitely many algebraic subgroups $\varphi(U)$ which implies that $\varphi(\PGL_{n+1}(\mathbb{K})) \subset \Aut(X \simeq \mathbb{A}^n)$ is an algebraic subgroup. Further, $\varphi(T) \subset \varphi(\PGL_{n+1}(\mathbb{K}))$ is a maximal subtorus that is isomorphic to $\mathbb{G}_m^n$ which means that $\varphi(\PGL_{n+1}(\mathbb{K}))$ is a simple algebraic group of rank $n$ that is isomorphic to $\PGL_{n+1}(\mathbb{K})$ as an abstract group. We conclude that $\varphi(\PGL_{n+1}(\mathbb{K}))$ is isomorphic to $\PGL_{n+1}(\mathbb{K})$ as an algebraic group. But this is not possible since the algebraic group $\PGL_{n+1}(\mathbb{K})$ does not act  regularly on $\mathbb{A}^n$ as the only closed subgroup $H$ of codimension $\le n$  of $\PGL_{n+1}(\mathbb{K})$ 
is a maximal parabolic  subgroup such that $\PGL_{n+1}(\mathbb{K})/H \simeq \mathbb{P}^n$.
We arrive to a contradiction with the isomorphism $X \simeq \mathbb{A}^n$. The proof follows.
\end{proof}

\begin{proof}[Proof of Theorem \ref{main3}]
Let $H \subset \Bir(X)$ be a maximal algebraic subtorus. By \cite[Theorem 1]{Ma} $X$ is birationally equivalent to $\mathbb{A}^l \times Z$, 
where $H$ acts on $\mathbb{A}^l$ with an open orbit and
 $Z \subset \mathbb{A}^r$ is an affine  variety with a trivial action of $H$.
 By Proposition \ref{Cremona} we can assume that $Z$ is positive dimensional. 
Assume there is an isomorphism $\varphi\colon \Bir(X) = \Bir(\mathbb{A}^l \times Z) \to \Aut(Y)$. 
Consider the maximal commutative subgroup $G$ of $\Bir(\mathbb{A}^l \times Z)$ of the form
\[
 \{ (x_1,\dots,x_l,z_1,\dots,z_r) \mapsto 
(f_1(z)x_1,\dots,f_l(z)x_l,g_1(z),\dots,g_r(z)) \mid f_i(z),g_i(z) \in \K(Z) \}
\]
that contains the commutative subgroup
\[
 \{ (x_1,\dots,x_l,z_1,\dots,z_r) \mapsto 
(f_1(z)x_1,\dots,f_l(z)x_l,z_1,\dots,z_r) \mid f_i(z) \in \K(Z) \},
\]
where the map
\[
Z \to Z \; \; (z_1,\dots,z_r) \mapsto 
(g_1(z),\dots,g_r(z))
\]
is a birational transformation of $Z$.

\begin{claim}\label{centralizer}
The subgroup 
$G \subset \Bir(\mathbb{A}^l \times Z)$
coincides with its centralizer.
\end{claim}

To prove this claim consider a birational transformation $\phi$ of $\mathbb{A}^l \times Z$ of the form
\[
 (x_1,\dots,x_l,z_1,\dots,z_r) \mapsto 
(F_1(x,z),\dots,F_l(x,z),G_1(x,z),\dots,G_r(x,z)),
\]
where $F_i(x,z),G_i(x,z) \in \K(\mathbb{A}^l \times Z)$, $x=(x_1,\dots,x_l)$, $z=(z_1,\dots,z_r)$ that commutes with each element from $G$. Hence, $\phi$ commutes with $T \subset G$, i.e., with all birational transformations of $\mathbb{A}^l \times Z$ of the form 
\[
 (x_1,\dots,x_l,z_1,\dots,z_r) \mapsto 
(t_1x_1,\dots,t_lx_l,z_1,\dots,z_r),\; \;   t_1,\dots,t_r \in \mathbb{K}^*.
\]
Direct computations show that
\[
 t_i F_i(t_1^{-1}x_1,\dots,t_l^{-1}x_l,z_1,\dots,z_r ) = F_i(x_1,\dots,x_l,z_1,\dots,z_r)
\]
and
\[
G_j(t_1^{-1}x_1,\dots,t_l^{-1}x_l,z_1,\dots,z_r ) = G_j(x_1,\dots,x_l,z_1,\dots,z_r)
\]
for all $t_1,\dots,t_r \in \mathbb{K}^*$, $i=1\dots,l$, $j=1,\dots,r$.
Therefore, 
$F_i(x,z) = h_i(z)x_i$ for some $h_i(z) \in \mathbb{K}(Z)$ and $G_j \in \mathbb{K}(Z)$. This proves the claim.

By \cite[Lemma 2.4]{LRU} $\varphi(G) \subset \Aut(Y)$ is a closed ind-subgroup and
 by \cite[Theorem B]{CRX19} (see also \cite[Corollary 3.2]{RvS21}) the connected component 
 $\varphi(G)^\circ \subset \Aut(Y)$ is the union of commutative algebraic subgroups. Since  $\varphi(G)^\circ \subset \varphi(G)$ is a countable index subgroup, there is an element $g=(f_1(z)x_1,x_2,\dots,x_l,z_1, \dots,z_r) \in G$ with non-constant $f_1$ such  that $\varphi(g)$ belongs to $\varphi(G)^\circ$. 
Since $\varphi(G)^\circ$ is the union of connected algebraic groups, $\varphi(g)$ belongs to a connected algebraic subgroup of $\varphi(G)^\circ$ and hence there exists $k > 0$ such that $\varphi(g)^k$ is a divisible element (see Section \ref{divisibleelements}). The element $g^k$ again has a form
$(\tilde{f}_1(z)x_1,x_2,\dots,x_l,z_1,\dots,z_r)$ with a non-constant $\tilde{f}_1 = \frac{r_1}{r_2}$, where $r_1, r_2 \in \mathcal{O}(Z)$. Moreover, $g^k \in G$ is not divisible. More precisely, without loss of generality we can assume that $r_1$ is non-constant and hence, there is no $h \in  G$ such that $h^{\deg r_1 + 1} = g^k$. Indeed, otherwise there would exist a rational function $s \in \K(Z)$ such that $s^{\deg r_1 + 1} = \tilde{f}_1 = \frac{r_1}{r_2}$ which is not the case. We get the contradiction which proves the theorem.
\end{proof}


\begin{thebibliography}{}

\bibitem{BD}
J. Blanc,  and A. Dubouloz, 
\textit{Affine surfaces with a huge group of automorphisms}, Int. Math. Res. Not.,  \textbf{2015}, Iss. 2, pp 422--459.

\bibitem{CRX19}  S. Cantat, A. Regeta, and J. Xie,
 \textit{Families of commuting automorphisms, and a
characterization of the affine space}, arXiv:1912.01567 to appear in Amer. J. Math.

\bibitem{Cor} Y. Cornulier,  \textit{Nonlinearity of some subgroups of the planar Cremona group}, arXiv:1701.00275.

\bibitem{FZ-uniqueness} H.~Flenner, M.~Zaidenberg, \emph{On the uniqueness of $\CC^*$-actions on affine surfaces},  Affine Algebraic Geometry, 97--111, Contemporary Mathematics \textbf{369}, Amer. Math. Soc. Providence, R.I., 2005.
\bibitem{FK} J.-P.~Furter,  H.~Kraft, \emph{On the geometry of the automorphism groups of affine varieties}, arXiv:1809.04175.

\bibitem{FL}
J.-P.~Furter, S.~Lamy,
\emph{Normal subgroup generated by a plane polynomial automorphism}, Transf. Groups \textbf{15}, no 3 (2010), 577--610. 
 

\bibitem{Ju42}
H. W. E. Jung, \textit{\"{U}ber ganze birationale Transformationen der Ebene},
J. Reine Angew. Math. 184 (1942), 161-174.


\bibitem{Kul53}
W. van der Kulk, \textit{On polynomial rings in two variables}, Nieuw Arch. Wiskunde (3) 1 (1953), 33-41. 

\bibitem{Kr}
H. Kraft, \emph{Fiber Bundles, Slice Theorem and Applications}, https://kraftadmin.wixsite.com/hpkraft.

\bibitem{LRU}
A. Liendo, A. Regeta and C. Urech, \emph{Characterization of affine surfaces with a torus action
by their automorphism groups}, arXiv:1805.03991,
to appear in Ann. Sc. Norm. Super. Pisa,  doi:10.2422/2036-2145.201905\_009.

\bibitem{Mar}
Z. S. Marciniak, \emph{A Note on Free Products of Linear Groups},  Proc. of the Amer. Math. Soc.
Vol. \textbf{94}, No. 1 (1985), pp. 46-48.

\bibitem{Ma}
H. Matsumura, \emph{On algebraic groups of birational transformations}, Atti Accad.
Naz. Lincei Rend. Cl. Sci. Fis. Mat. Natur. (8) \textbf{34} (1963), 151-155.

\bibitem{RvS21} 
A. Regeta and I. van Santen,  \textit{Characterizing Affine Toric Varieties via the Automorphism Group}, 
 arXiv:2112.04784 

\bibitem{Sh66}
I.~R. Shafarevich, \emph{On some infinite-dimensional groups}, Rend. Mat. e
  Appl. (5) \textbf{25} (1966), no.~1-2, 208--212.
  
  \bibitem{Sr80}  J.P. Serre, \textit{Trees}, Springer,Berlin-Heidelberg-New York, 1980.

\end{thebibliography}
\end{document}